\documentclass[reqno]{amsart}

\usepackage{amsfonts,amssymb,amsmath,latexsym,indentfirst}
\usepackage{fontenc,enumerate}
\usepackage{graphics}
\usepackage{color,graphicx}
\usepackage[pdftex]{hyperref}
\usepackage{latexsym,wasysym,mathrsfs,enumerate}
\DeclareMathAlphabet{\mathpzc}{OT1}{pzc}{m}{it}

\newcommand{\vr}{\varrho}

\newcommand{\rr}{{\mathbb R}}

\newcommand{\lam}{{\lambda}}
\newcommand{\ff}{\varphi}

\newcommand{\h}{\mathscr{H}}

\newcommand{\A}{\mathscr{A}}

\newcommand{\dx}{D_x^{1/2}}
\newcommand{\hf}{{\dot{H}^{(1/2,1)}}}
\newcommand{\ihf}{{{H}^{(1/2,1)}}}
\newcommand{\lt}{{L^2}}

\newcommand{\fim}{\hfill$\square$\\ \\}

\begin{document}
\title[Sharp constant of an anisotropic GN inequality]
{Sharp constant of an anisotropic Gagliardo-Nirenberg-type inequality and
applications}



\maketitle

\numberwithin{equation}{section}
\newtheorem{theorem}{Theorem}[section]
\newtheorem{lemma}[theorem]{Lemma}
\newtheorem{proposition}[theorem]{Proposition}
\newtheorem{corollary}[theorem]{Corollary}
\newtheorem{definition}[theorem]{Definition}
\newtheorem{remark}[theorem]{Remark}

\begin{center}
 {\large \textbf{Amin Esfahani}}
\end{center}

\begin{center}
  {School of Mathematics and Computer Science,
Damghan University, Damghan 36715-364, Iran;  and School of Mathematics,
Institute for Research in Fundamental Sciences  (IPM),   Tehran 19395-5746,
Iran.\\ E-mail: amin@impa.br, esfahani@du.ac.ir }
 \end{center}

\vskip.5cm

\begin{center}
 {\large \textbf{Ademir Pastor }}
\end{center}

\begin{center}
  { IMECC--UNICAMP,
Rua S\'ergio Buarque de Holanda, 651, Cidade Universit\'aria,
13083-859, Campinas--SP, Brazil.\\
 E-mail: apastor@ime.unicamp.br}
 \end{center}

\begin{abstract}
In this paper we   establish the best constant  of  an anisotropic
Gagliardo-Nirenberg-type inequality related to the
Benjamin-Ono-Zakharov-Kuznetsov  equation. As an application of our results,
we prove the uniform bound of solutions for such a equation in the energy
space.
\end{abstract}

\section{Introduction}

This paper is concerned with the best constant of the following
two-dimensional anisotropic \emph{Gagliardo-Nirenberg}-type inequality
\begin{equation}\label{inhomo-embed}
\|u\|_{L^{p+2}}^{p+2}\leq \vr{\|u\|}_{L^2}^{(4-p)/2}{\|\dx
u\|}_{L^2}^{p}{\|u_y\|}_{L^2}^{p/2}, \qquad u=u(x,y)\in H^{(1/2,1)},
\end{equation}
where $0< p<4$, $\vr$ is a positive constant, $L^q:=L^q(\rr^2)$ is the usual
Lebesgue space, $D_x^{1/2}$ represents the $1/2$-derivative operator in the
$x$-variable defined via its Fourier transform as
$\widehat{D_x^{1/2}u}(\xi,\eta)=|\xi|^{1/2}\widehat{u}(\xi,\eta)$, and
$H^{(1/2,1)}:=\ihf(\rr^2)$ denotes the fractional Sobolev-Liouville space
(see \cite{lizorkin}) as the closure of $C_0^\infty(\rr^2)$ endowed  with the
norm
\[
\|u\|_\ihf^2=\|u\|_\lt^2+\|\dx u\|_{\lt}^2+\|u_y\|_\lt^2.
\]

Inequality \eqref{inhomo-embed} is closely related with  the two-dimensional
generalized Benjamin-Ono-Zakharov-Kuznetsov (BO-ZK henceforth) equation
\begin{equation}\label{bo-zk}
u_t-\h u_{xx}+ u_{xyy}+\partial_x(u^{p+1})=0, \quad (x,y)\in\rr^2,\;\;t>0,
\end{equation}
where $\h$ stands for the Hilbert transform in the $x$-variable, defined by
\[
\mathscr{H}u(x,y,t)=\mathrm{p.v.}\frac{1}{\pi}\int_\rr\dfrac{u(z,y,t)}{x-z}\;dz.
\]
Indeed, in \cite{aaj}, by using \eqref{inhomo-embed}, the authors have
studied the existence   of solitary-wave solutions. It was proved that a
nontrivial solitary-wave solution of the form $u(x,y,t)=\ff(x-t,y)$ (with
velocity $c=1$) of \eqref{bo-zk}   exists if $0<p<4$. Assuming that $\ff$ has
a suitable decay at infinity, one see that $\ff$ should satisfy
 \begin{equation}\label{bozk-2}
 -\ff+\ff^{p+1}-\h\ff_x+\ff_{yy}=0.
 \end{equation}
In order to show the existence of solitary waves, the authors in \cite{aaj}
applied  the concentration-compactness principle \cite{lions} for the
following minimization problem
\begin{equation} \label{minimiz-constr-lp}
I_\lam=\inf\left\{ I(\ff) \;;\;\ff\in\ihf\;,\;J(\ff)=
\int_{\rr^2}\ff^{p+2}\;dxdy=\lam>0\right\},
\end{equation}
where $\lambda$ is a prescribed number and
\[
I(\ff)=\frac{1}{2}\int_{\rr^2}\left(\ff^2+\ff\h\ff_x+\ff_y^2\right)\; dxdy
=\frac{1}{2}\|\ff\|_{\ihf}^2.
\]
Inequality \eqref{inhomo-embed} shows, in particular, that $\ihf$ is
continuously embedded in $L^{p+2}$. Hence, the minimization problem
\eqref{minimiz-constr-lp} is well-defined.

\begin{remark}
Of course, one can consider solitary-wave solutions of the form
$u(x,y,t)=\ff(x-ct,y)$. In this case, such solutions exists for any $c$
positive (see \cite{aaj}).
\end{remark}

\begin{remark}\label{pinremark}
In order to functional $J$ be well-defined for all $u\in\ihf$, we assume here
and throughout the paper that $p=k/\ell$, where $k$ and $\ell$ are relatively
 prime integer numbers and $\ell$ is odd.
\end{remark}

Sharp constant for the Gagliardo-Nirenberg inequality
\[
\|u\|_{L^{p+2}(\rr^n)}^{p+2}\leq K_{\rm best}^{p+2}\|\nabla
u\|_{L^{2}(\rr^n)}^{np/2}\|u\|_{L^{2}(\rr^n)}^{2+p(2-n)/2}
\]
was first studied in Nagy \cite{na} in the case $n=1$ and then for all
$n\geq2$ (with $0<p<4/(n-2)$) in Weienstein \cite{we}.  The sharp constant
was obtained in terms of the ground state solution of the semilinear elliptic
equation
\[
\frac{pn}{4}\Delta\psi-\left(1+\frac{p}{4}(2-n)\right)\psi+\psi^{p+1}=0.
\]
More precisely,
\[
K_{\rm best}^{p+2}=\frac{p+2}{2\|\psi\|_{L^{2}(\rr^n)}^2}.
\]
Since then much effort has been expended on the study  of
Gagliardo-Nirenberg-type inequalities and its best constants (see, for
instance, \cite{ABLS,bfv,bbm,bm,cfl,dpv,mo,we} and references therein). Such
a effort can be justified in view of the crucial role of these inequalities
in the study of global well-posedness of the Cauchy problem associated with
several equations (see \cite{ABLS,cfl,farah-lp,farah-lp1,flp,HR,hr1,KM,mm,we}
and references therein). In many examples (especially for critical and
supercritical nonlinearities) the dichotomy ``global well-posedness $\times$
finite time blow up'' can be described using the best constant of a
Gagliardo-Nirenberg-type inequality.

Equation \eqref{bo-zk} was introduced in \cite{jcms}, \cite{lmsv} as a model
to describe the electromigration in thin nanoconductors on a dielectric
substrate. The BO-ZK equation \eqref{bo-zk} can also be viewed as a
two-dimensional generalization of the Benjamin-Ono (BO henceforth) equation
\begin{equation}\label{bo}
u_t-\h u_{xx}+\partial_x(u^{p+1})=0, \quad x\in\rr,\;\;t>0,
\end{equation}
which appears as a model for long internal gravity waves in deep stratified
fluids (see \cite{be}). It is well-known (see, for instance,  \cite{be} or
\cite{bl}) that solitary-wave solutions of the BO equation has an algebraic
decay at infinity. Thus, it is expected that solitary waves of \eqref{bo-zk}
has an algebraic decay in the propagation direction and, in view of the
second order derivative, an exponential decay in the transverse direction.
This was confirmed in \cite{aaj}. From the physical viewpoint this
anisotropic behavior implies that solitary waves has a limited stability
range e decay into radiation outside this range (see \cite{lmsv}).

The Cauchy problem associated with \eqref{bo-zk} was considered in \cite{cp}, \cite{cp1},
\cite{pams}, \cite{aaj}. In particular, local well-posedness was established
in $H^s(\rr^2)$, $s>2$ (see Theorem \ref{local} below). In \cite{aa,aaj} was
also demonstrated that a solitary-wave solution (with arbitrary positive
velocity) is nonlinearly stable if $0<p<4/3$ and nonlinearly unstable if
$4/3<p<4$. Other properties of the solutions, including unique continuation
principles, were also proved in \cite{cp} and \cite{blms}.

It should be noted that $p=4/3$ is a ``critical value'' for \eqref{bo-zk}. We
present two reasons for this nomenclature. The first one is related with the
orbital stability of solitary waves: as we already said, solitary waves are
stable if $0<p<4/3$ and unstable if $4/3<p<4$ (we do not know if they are
stable or not for $p=4/3$). The second one is related with the scaling
argument: if $u$ solves  \eqref{bo-zk} with initial data $u_0$ then
$$
u_\lambda(x,y,t)=\lambda^{2/p}u(\lambda^2x,\lambda y, \lambda^4t)
$$
also solves \eqref{bo-zk} with initial data
$u_\lambda(x,y,0)=\lambda^{2/p}u_0(\lambda^2x,\lambda y)$, for any
$\lambda>0$. As a consequence, if
${\dot{H}^{s_1,s_2}}:={\dot{H}^{s_1,s_2}}(\rr^2)$ denotes the homogeneous
anisotropic Sobolev space, we have
$$
\|u_\lambda(\cdot,\cdot,0)\|_{\dot{H}^{s_1,s_2}}=\lambda^{2s_1+s_2+2/p-3/2}\|u_0\|_{\dot{H}^{s_1,s_2}}.
$$
Thus, $L^2$ is the scale-invariant Sobolev spaces for the BO-ZK equation if
and only if $p=4/3$.

In order to describe our main result in the present paper, let us define
\[
S(u)=\frac{1}{2}\|u\|_\ihf^2-\frac{1}{p+2}J(u),
\]
where $J$ is given in \eqref{minimiz-constr-lp}. We  recall that a solution
$\ff\in\ihf$ of \eqref{bozk-2} is called a \textit{ground state}, if $\ff$
minimizes the action $S$ among all solutions of \eqref{bozk-2}. Our main
theorem reads as follows.

\begin{theorem}\label{teo-best}
Let $0<p<4$. Then the best   constant $\vr$ in the fractional
Gagliardo-Nirenberg inequality \eqref{inhomo-embed} is such that
\begin{equation} \label{besng}
\vr^{-1}=\frac{4-p}{2(p+2)}\left(\frac{p}{4-p}\right)^{3p/4}2^{p/2}\|\ff\|_\lt^p
= \frac{4-p}{2(p+2)}\left(\frac{p}{4-p}\right)^{p/4}(2d)^{p/2},
\end{equation}
where $\ff$ is a ground state solution of \eqref{bozk-2} and
$$
d=\inf\{S(u);\;u\in\Lambda\},
$$
with $\Lambda=\{u\in\ihf;\;u\neq0,\;S'(u)=0\}$.
\end{theorem}

\begin{remark}
Provided we know the existence of positive ground state solutions, Theorem \ref{teo-best} still holds if $p\in(0,4)$ is not a rational number  (see Remark \ref{pinremark}).
\end{remark}

We prove Theorem \ref{teo-best} following some ideas developed in \cite{cfl}
where the sharp constant for a  Gagliardo-Nirenberg-type inequality related
with Kadomtsev-Petviashvili-type equations was established. Because we are
dealing with anisotropic spaces, the classical method used in \cite{we}
cannot be directly used. This is overcame by using scaling arguments.

\begin{remark}
Uniqueness of ground state solutions for \eqref{bozk-2} seems to be a very
interesting and challenging issue. In view of the anisotropic nature of
\eqref{bozk-2}, it is not clear if the recent theory developed in \cite{fl}
and \cite{fls1} can be applied. Note, however, from the second equality in
\eqref{besng}, that $\vr$ does not depend on the choice of the ground state
(if there are many).
\end{remark}

As an application of inequality \eqref{inhomo-embed},  we shall prove the
uniform bound of solutions of \eqref{bo-zk}. More precisely, in the
subcritical and critical regimes, we have the following.

\begin{theorem}\label{uniform}
Let $u_0\in H^s(\rr^2)$, $s>2$, and $u\in C([0,T);H^s(\rr^2))$ be the
solution of \eqref{bo-zk}, associated with the initial value $u_0$. Then
$u(t)$ is uniformly bounded in $\ihf$, for $t\in [0,T)$, if one of the
following conditions hold:
\begin{enumerate}[(i)]
\item $0<p<4/3$; \item $p=4/3$ and
\begin{equation}\label{uni134}
\|u_0\|_\lt^4<\left(\frac{4}{27}\right)\|\ff\|_\lt^4,
\end{equation}
where $\ff$ is a ground state of \eqref{bozk-2}.
\end{enumerate}
\end{theorem}

In the supercritical regime, that is, for  $4/3<p<4$, additional conditions
on the initial data must be imposed. More precisely, we prove the following.

\begin{theorem}\label{uniform1}
Assume $4/3<p<4$. Suppose that $u_0\in H^s(\rr^2)$, $s>2$, satisfies
\begin{equation}\label{cond1}
\|u_0\|_{\lt}^{2(4-p)}\|u_0\|_{\dot{H}^{(1/2,1)}}^{2(3p-4)}<\left(\frac{4}{27}\right)^p
\|\ff\|_{\lt}^{2(4-p)}\|\ff\|_{\dot{H}^{(1/2,1)}}^{2(3p-4)}, \quad E(u_0)>0,
\end{equation}
and
\begin{equation}\label{cond2}
\|u_0\|_{\lt}^{2(4-p)}E(u_0)^{3p-4}<\left(\frac{4}{27}\right)^p
\|\ff\|_{\lt}^{2(4-p)}E(\ff)^{3p-4},
\end{equation}
where $\ff$ is a ground state solution of \eqref{bozk-2}, $E$ is the energy
defined in \eqref{third-condition}, and ${\dot{H}^{(1/2,1)}}$ is the
homogeneous fractional Sobolev-Liouville space with the norm
\[
\|u\|_{\dot{H}^{(1/2,1)}}^2=\|\dx u\|_{\lt}^2+\|u_y\|_\lt^2.
\]
Let $u\in C([0,T);H^s(\rr^2))$ be the solution of \eqref{bo-zk}, associated
with the initial value $u_0$. Then $u(t)$ is uniformly bounded in $\ihf$, for
$t\in [0,T)$. In addition, we have the bound
\begin{equation}\label{cond3}
\|u_0\|_{\lt}^{2(4-p)}\|u(t)\|_{\dot{H}^{(1/2,1)}}^{2(3p-4)}<\left(\frac{4}{27}\right)^p
\|\ff\|_{\lt}^{2(4-p)}\|\ff\|_{\dot{H}^{(1/2,1)}}^{2(3p-4)}.
\end{equation}
\end{theorem}

The proofs of Theorems \ref{uniform} and \ref{uniform1} will follow taking
into account the exact value of $\varrho$ in \eqref{besng}. Uniform bound in
general is not a triviality and relies on different aspects of the
differential equation in hand. Here, the conservation of the mass and the
energy play a crucial role.

\begin{remark}
It is easy to see that if $s>2$ and $u\in H^s(\rr^2)$, then  $u\in \ihf$.
Although we do not know about the local well-posedness in $\ihf$, the uniform
bounds in Theorems \ref{uniform} and \ref{uniform1} could lead a local
well-posedness result to a global one in the energy space.
\end{remark}

The remainder of the paper is organized as follows. In Section
\ref{bcons-sec} we prove that inequality \eqref{inhomo-embed} holds for some
positive constant $\vr$ and recall some useful properties of the ground state
solutions of \eqref{bozk-2}. In Section \ref{sec3} we prove Theorem
\eqref{teo-best} and establish the sharp constant \eqref{besng}. Finally, in
Section \ref{sec4}, we present the proofs of  Theorems  \ref{uniform} and
\ref{uniform1}

\section{The inequality \eqref{inhomo-embed} and properties of ground states}\label{bcons-sec}

We start this section by proving inequality \eqref{inhomo-embed}. Roughly
speaking, it follows as an application of the usual H\"older and Minkowski
inequalities combined with the one-dimensional fractional Gagliardo-Nirenberg
inequality:
\begin{equation}\label{gn1d}
\|f\|_{L^{r}(\rr)}^{r}\le C\,\|D_x^{\beta/2}
f\|_{L^2(\rr)}^{(r-2)/\beta}\|f\|_{L^2(\rr)}^{(2+r(\beta-1))/\beta},
\end{equation}
which holds for all $r\geq2$, $\beta\geq1$, and $f\in H^{\beta/2}(\rr)$ (see,
for instance, \cite{ABLS}). Here, for functions $f=f(x)$ of one real
variable, $D^{\beta/2}_x$ denotes the operator defined via Fourier transform
as $\widehat{D^{\beta/2}_xf}(\xi)=|\xi|^{\beta/2}\widehat{f}(\xi)$. In
addition, the smallest constant $C=C_{r,\beta}$ for which \eqref{gn1d} holds
is given by
\begin{equation}\label{gn1dbest}
C_{r,\beta}=\frac{r\beta}{2+r(\beta-1)}\left[ \left( \frac{2+r(\beta-1)}{r-2}
\right)^{1/\beta} \frac{1}{\|\Psi\|^2_{L^2(\rr)}}\right]^{(r-2)/2},
\end{equation}
where $\Psi$ is a solution of
$$
D^{\beta}\Psi+\Psi-|\Psi|^{r-2}\Psi=0.
$$

Now we are able to prove inequality \eqref{inhomo-embed}.

\begin{proposition}\label{bozk-ineq}
Let $0<p<4$. Then there exists $\vr>0$ such that inequality
\eqref{inhomo-embed} holds, for all $u\in\ihf$.
\end{proposition}
\begin{proof} The lemma is established for $C_0^\infty(\rr^2)$-functions  and then
limits are taken to complete the proof. By \eqref{gn1d}, with $\beta=2$,  we
deduce the existence of $C>0$ such that
\[
\|u(x,\cdot)\|_{L^{p+2}(\rr)}^{p+2}\leq
C\|u(x,\cdot)\|_{L^2(\rr)}^{\frac{p+4}{2}}\|u_y(x,\cdot)\|_{L^2(\rr)}^{\frac{p}{2}}.
\]
From this point on, the constant $C>0$ may vary from line to line. By using
the H\"older and Minkowski inequalities, it follows that
\begin{equation}\label{ineq-0}
\begin{split}
\|u\|_{L^{p+2}(\rr^2)}^{p+2}&
\leq C\int_\rr\|u(x,\cdot)\|_{L^2(\rr)}^{\frac{p+4}{2}}\;\|u_y(x,\cdot)\|_{L^2(\rr)}^{\frac{p}{2}}\; dx\\
&\leq C
\left\|\|u\|_{L^2(\rr_y)}\right\|_{L^{\frac{2(p+4)}{4-p}}(\rr_x)}^{\frac{p+4}{2}}\;
\|u_y\|_{L^2(\rr^2)}^{\frac{p}{2}}\\
&\leq C
\left\|\|u\|_{L^{\frac{2(p+4)}{4-p}}(\rr_x)}\right\|_{L^{2}(\rr_y)}^{\frac{p+4}{2}}\;
\|u_y\|_{L^2(\rr^2)}^{\frac{p}{2}}.
\end{split}
\end{equation}
Another application of \eqref{gn1d}, with $\beta=1$, reveals that
\begin{equation}\label{ineq-01}
\begin{split}
\|u\|_{L^{p+2}(\rr^2)}^{p+2}&
\leq C\left(\int_\rr\|\dx u(\cdot,y)\|_{L^2(\rr)}^{\frac{4p}{p+4}}\|u(\cdot,y)\|_{L^2(\rr)}^{\frac{2(4-p)}{p+4}}dy
\right)^{\frac{p+4}{4}}\|u_y\|_{L^2(\rr^2)}^{\frac{p}{2}}\\
&\leq C\|\dx
u\|_{\lt}^p\|u\|_{\lt}^\frac{4-p}{2}\|u_y\|_{L^2(\rr^2)}^{\frac{p}{2}}.
\end{split}
\end{equation}
This completes the proof.
\end{proof}

To proceed, we recall that the  existence of ground state solutions for
\eqref{bozk-2} was established in \cite{aaj}. In what follow in this section,
we prove some properties of the ground states, which will be useful to prove
Theorem \ref{teo-best}. Some of them were given in \cite{aaj}, but for the
sake of completeness we bring some details. Let us start by observing that $\mathscr{H}(x\ff_x)=x\mathscr{H}(\ff_x)$. Thus, since $\mathscr{H}$ is a skew-symmetric operator, we have
\[
-\int_{\rr^2}x\ff_x\mathscr{H}(\ff_x)\;dxdy=\int_{\rr^2}\ff_x\mathscr{H}(x\ff_x)\;dxdy=\int_{\rr^2}x\ff_x\mathscr{H}(\ff_x)\;dxdy,
\]
which implies that
\begin{equation}\label{zeroH}
\int_{\rr^2}x\ff_x\mathscr{H}(\ff_x)\;dxdy=0.
\end{equation}

\begin{lemma}\label{lemsapre}
Let $\ff$ be a ground state solution of \eqref{bozk-2}. Then,
\begin{itemize}
  \item[(i)] ${\displaystyle J(\ff)=\frac{p+2}{p}\|D^{1/2}_x\ff\|^2_\lt}$,
  \item[(ii)] ${\displaystyle \|\ff\|^2_\lt=\frac{4-p}{2p}\|D^{1/2}_x\ff\|^2_\lt}$,
  \item[(iii)] ${\displaystyle \|\ff_y\|^2_\lt=\frac{1}{2}\|D^{1/2}_x\ff\|^2_\lt}$.
\end{itemize}
\end{lemma}
\begin{proof}
First we recall that ground state solutions are $C^\infty$ and together with all its derivatives are bounded and tend to zero at infinity. In addition, there is a constant $\sigma>0$ such that, for any ground state $\ff$, $|x|^se^{\sigma|y|}\ff(x,y)\in L^1(\rr^2)\cap L^\infty(\rr^2)$, $s\in[0,3/2)$ (see Theorems 4.7 and 5.9 in \cite{aaj}). This is enough to justify the calculations to follow.
We multiply equation \eqref{bozk-2} by $\ff$, $x\ff_x$, and $y\ff_y$,
respectively, integrate over $\rr^2$, use \eqref{zeroH} and elementary properties of the
Hilbert transform together with integration by parts to get
\begin{gather}
\int_{\rr^2}\left(\ff^2+\ff\mathscr{H}\ff_x+\ff_y^2
-\ff^{p+2}\right)\;dxdy=0\label{2bozk-nonex-1},\\
\int_{\rr^2}\left(\ff^2+\ff_y^2-\frac{2}{p+2}
\ff^{p+2}\right)\;dxdy=0\label{2bozk-nonex-2},\\
\int_{\rr^2}\left(\ff^2+\ff\mathscr{H}\ff_x-\ff_y^2-
\frac{2}{p+2}\ff^{p+2}\right)\;dxdy=0.\label{2bozk-nonex-3}
\end{gather}
Subtracting \eqref{2bozk-nonex-2} from \eqref{2bozk-nonex-1} we obtain
\begin{equation}
\int_{\rr^2}\left(\ff\mathscr{H}\ff_x
-\frac{p}{p+2}\ff^{p+2}\right)\;dxdy=0\label{2bozk-nonex-5}.
\end{equation}
This proves (i) because
$\int_{\rr^2}\ff\mathscr{H}\ff_xdxdy=\|D^{1/2}_x\ff\|^2_\lt$. To prove (ii),
we add \eqref{2bozk-nonex-2} and \eqref{2bozk-nonex-3} to have
\begin{equation}
\int_{\rr^2}\left(\ff^2+\frac{1}{2}\ff\mathscr{H}\ff_x
-\frac{2}{p+2}\ff^{p+2}\right)\;dxdy=0\label{2bozk-nonex-6}.
\end{equation}
From \eqref{2bozk-nonex-6} and using part (i) we deduce
$$
\|\ff\|^2_\lt=\left(\frac{2}{p}-\frac{1}{2}\right)\|D^{1/2}_x\ff\|^2_\lt=\frac{4-p}{2p}\|D^{1/2}_x\ff\|^2_\lt.
$$
Finally, using \eqref{2bozk-nonex-1} and parts (i) and (ii) we get (iii). The
proof of the lemma is thus completed.
\end{proof}

\begin{lemma}\label{equival}
Let
\[
K(u)=\dfrac{1}{2}\left(\|u\|_\lt^2+\|u_y\|_\lt^2\right)-\frac{1}{p+2}J(u).
\]
Assume that $\ff$ is a ground state solution of \eqref{bozk-2}. Then,
$K(\ff)=0$ and $\ff$ minimizes the functional $I$ among all solutions of
\eqref{bozk-2}.
\end{lemma}
\begin{proof}
Let $u\in\ihf$ be a solution of \eqref{bozk-2}. Note that the properties
determined in Lemma \ref{lemsapre} does not depend on the fact that $\ff$ is
a ground state but only on the fact the $\ff$ is a solution of
\eqref{bozk-2}. Thus, the same properties hold for $u$ and
$$
K(u)=\frac{1}{2}\left(\frac{4-p}{2p}+\frac{1}{2}\right)\|D_x^{1/2}u\|^2_\lt-\frac{1}{p}\|D_x^{1/2}u\|^2_\lt=0.
$$
In particular we have $K(\ff)=0$.

By definition it is inferred that
$S(u)=K(u)+\frac{1}{2}\|D_x^{1/2}u\|^2_\lt$. By Taking into account that
$\ff$ is a ground state, we have
$$
\frac{1}{2}\|D_x^{1/2}\ff\|^2_\lt=S(\ff)\leq
S(u)=\frac{1}{2}\|D_x^{1/2}u\|^2_\lt.
$$
This shows that $\ff$ minimizes $\|D_x^{1/2}u\|^2_\lt$ among all solutions of
\eqref{bozk-2}. But since,
$$
I(u)=\frac{1}{2}\left(1+\frac{2}{p}\right)\|D_x^{1/2}u\|^2_\lt,
$$
we then deduce
$$
I(\ff)=\frac{1}{2}\left(1+\frac{2}{p}\right)\|D_x^{1/2}\ff\|^2_\lt\leq
\frac{1}{2}\left(1+\frac{2}{p}\right)\|D_x^{1/2}u\|^2_\lt=I(u).
$$
This completes the proof.
\end{proof}

\begin{lemma} \label{equivala1}
Let $\ff$ be a ground state solution of \eqref{bozk-2}. Assume that
$u\in\ihf$ satisfies $J(u)=J(\ff)$. Then, $I(\ff)\leq I(u)$.
\end{lemma}
\begin{proof}
Let $\lambda=J(\ff)$. Let $v$ be a minimum of the minimization problem
\eqref{minimiz-constr-lp}. Since $I(v)\leq I(u)$ for all $u\in\ihf$
satisfying $J(u)=\lambda$, it suffices to show that
\begin{equation}\label{a11}
I(\ff)\leq I(v).
\end{equation}
Because $v$ minimizes $I_\lambda$, we obtain
\begin{equation}\label{a12}
I_\lambda=I(v)\leq I(\ff).
\end{equation}
Moreover, there exists a positive Lagrange multiplier $\theta$ such that
\begin{equation}\label{a13}
v+\mathscr{H}v_x-v_{yy}=\theta v^{p+1}.
\end{equation}
Multiplying \eqref{a13} by $v$,  integrating over $\rr^2$ and using
\eqref{a12} yield
$$
\theta \lambda=\theta J(v)=2I(v)\leq 2I(\ff)=J(\ff)=\lambda.
$$
This shows that $0<\theta\leq1$. Now define $w=\theta^{1/p}v$. It is easy to see
that $w$ is a solution of \eqref{bozk-2}. Therefore, from Lemma \ref{equival}
and \eqref{a12},
$$
I(\ff)\leq I(w)=\theta^{2/p}I(v)\leq \theta^{2/p}I(\ff).
$$
With this last inequality we then conclude that $\theta=1$ and the proof is
completed.
\end{proof}

\begin{lemma}\label{equivala2}
Let $\ff$ be a ground state solution of \eqref{bozk-2}. Then
$$
\inf\{\|D_x^{1/2}u\|^2_\lt;\;\; u\in\ihf,
u\neq0,K(u)=0\}=\|D_x^{1/2}\ff\|^2_\lt,
$$
where $K$ is defined in Lemma \ref{equival}.
\end{lemma}
\begin{proof}
Let $u\in\ihf$ be such that $u\neq0$ and $K(u)=0$. From the definition of $K$
we have $J(u)>0$. Define
$$
u_\mu(x,y)=u\left(\frac{x}{\mu},y\right), \qquad \mu=\frac{J(\ff)}{J(u)}.
$$
A straightforward calculation reveals that $J(u_\mu)=J(\ff)$ and
$K(u_\mu)=0$. Since $u_\mu\in\ihf$, Lemma \ref{equivala1} implies that
$I(\ff)\leq I(u_\mu)$. Observe that
$$
I(v)=K(v)+\frac{1}{p+2}J(v)+\frac{1}{2}\|D_x^{1/2}v\|^2_\lt, \qquad {\rm for\; all}\;\;
v\in \ihf.
$$
Therefore,
$$
K(\ff)+\frac{1}{p+2}J(\ff)+\frac{1}{2}\|D_x^{1/2}\ff\|^2_\lt \leq
K(u_\mu)+\frac{1}{p+2}J(u_\mu)+\frac{1}{2}\|D_x^{1/2}u_\mu\|^2_\lt.
$$
The facts that $K(\ff)=K(u_\mu)=0$ and $J(\ff)=J(u_\mu)$ then imply the
desired because $\|D_x^{1/2}u_\mu\|^2_\lt=\|D_x^{1/2}u\|^2_\lt$. The proof is
thus completed.
\end{proof}

\section{Proof of Theorem  \ref{teo-best}}\label{sec3}

In this section we will prove Theorem \ref{teo-best}. First we show that
\[
\vr^{-1}\geq
\frac{4-p}{2(p+2)}\left(\frac{p}{4-p}\right)^{p/4}\|\dx\ff\|_\lt^p.
\]
Let  $u\in\ihf$ be such that $u\neq0$ and $J(u)>0$. Choose positive real
constants $\kappa,\xi$, and $\mu$ such that
\[
\omega(x,y)=\kappa u(\xi x,\mu y)
\]
satisfies
\begin{eqnarray}
\|\omega_y\|_\lt=\kappa\mu^{1/2}\xi^{-1/2} \|u_y\|_\lt=\frac{1}{\sqrt2}\|\dx \ff\|_\lt,\label{comp-1}\\
J(\omega)=\kappa^{p+2}\mu^{-1}\xi^{-1}J(u)=\frac{p+2}{p}\|\dx\ff\|_\lt^2\label{comp-2}\\
\|\omega\|_\lt =\kappa
\mu^{-1/2}\xi^{-1/2}\|u\|_\lt=\left(\frac{4-p}{2p}\right)^{1/2}\|\dx\ff\|_\lt.\label{comp-3}
\end{eqnarray}

A straightforward algebraic computation reveals that such a choice is always
possible. In particular, gathering together identities \eqref{comp-1},
\eqref{comp-2}, and \eqref{comp-3} give
\begin{equation}\label{ai1}
\kappa^p=\frac{2(p+2)}{4-p}\|u\|_\lt^2\left(J(u)\right)^{-1}
\end{equation}
and
\begin{equation}\label{ai2}
\mu^{-2}=\frac{4-p}{p}\|u_y\|_\lt^2\|u\|_\lt^{-2}.
\end{equation}
Hence, using Plancherel's identity, \eqref{ai1} and \eqref{ai2} we get
\begin{equation}\label{comp-4}\begin{split}
\|\dx
\omega\|_\lt^2&=\left(\frac{4-p}{p}\right)^{1/2}\left(\frac{2(p+2)}{4-p}\right)^{2/p}
\|\dx u\|_\lt^2 \|u_y\|_\lt \|u\|_\lt^{\frac{4-p}{p}} (J(u))^{-2/p}.
\end{split}
\end{equation}
By using \eqref{comp-1}-\eqref{comp-3}  it is readily seen that
$K(\omega)=0$. Therefore, Lemma \ref{equivala2} implies
\begin{equation}\label{ai3}
\|\dx\omega\|_\lt\geq\|\dx \ff\|_\lt.
\end{equation}
On the other hand, observe that
\[
\varrho^{-1}=\inf\left\{\A(u);\;u\in\ihf,\;u\neq0,\;J(u)>0\right\},
\]
where
\[
\A(u)=\|\dx u\|_\lt^p\|u_y\|_\lt^{p/2}\|u\|_\lt^{(4-p)/2}J(u)^{-1}.
\]
Consequently, it follows from \eqref{comp-4} and \eqref{ai3} that
\[
\A(u)\geq\frac{4-p}{2(p+2)}\left(\frac{p}{4-p}\right)^{p/4}\|\dx\ff\|_\lt^{p}.
\]
Since $u$ is arbitrary, it is   concluded that

\begin{equation}\label{ai4}
\varrho^{-1}
\geq\frac{4-p}{2(p+2)}\left(\frac{p}{4-p}\right)^{p/4}\|\dx\ff\|_\lt^{p}.
\end{equation}

Next we prove the
\[
\varrho^{-1}
\leq\frac{4-p}{2(p+2)}\left(\frac{p}{4-p}\right)^{p/4}\|\dx\ff\|_\lt^{p}.
\]
Indeed, since $\ff\neq0$ and $J(\ff)>0$ we have
\begin{equation}\label{ai5}
\varrho^{-1}\leq \A(\ff).
\end{equation}
An application of Lemma \ref{lemsapre} infers  that
\begin{equation}\label{ai6}
\A(\ff)=\frac{4-p}{2(p+2)}\left(\frac{p}{4-p}\right)^{p/4}\|\dx\ff\|_\lt^{p}.
\end{equation}
Gathering together \eqref{ai5} and \eqref{ai6} and combining the result with
\eqref{ai4} we get
\[\begin{split}
\varrho^{-1}=\frac{4-p}{2(p+2)}\left(\frac{p}{4-p}\right)^{p/4}\|\dx\ff\|_\lt^{p}
\end{split}\]
Using Lemma \ref{lemsapre} we then deduce
\[\begin{split}
\varrho^{-1}=\frac{4-p}{2(p+2)}\left(\frac{p}{4-p}\right)^{3p/4}2^{p/2}\|\ff\|_\lt^p.
\end{split}\]
Finally, it is obvious that $d\leq S(\ff)$. On the other hand, if $u\in\ihf$
satisfies $S'(u)=0$ then $u$ is a solution of \eqref{bozk-2}, which implies
that $S(\ff)\leq S(u)$ and, hence, $S(\ff)\leq d$. Since Lemma \ref{lemsapre}
gives $S(\ff)=\frac{1}{2}\|\dx\ff\|_\lt^2$, the second equality in Theorem
\ref{teo-best} is thus proved.\fim

In view of \eqref{gn1dbest} we can prove the lower bound for the $L^2$-norm
of the solitary waves.
\begin{corollary}
If $\varphi$ is a nontrivial solution of \eqref{bozk-2}, then
\begin{equation}\label{lower}
\|\varphi\|_{L^2(\rr^2)}\geq\|\psi_2\|_{L^2(\rr)}\|\psi_1\|_{L^2(\rr)},
\end{equation}
where $\psi_2$ is a solution of
\begin{equation}
-\psi_2''+\psi_2-\psi_2^{p+1}=0
\end{equation}
and $\psi_1$ is a solution of
\begin{equation}
\mathscr{H}\psi_1+\psi_1-\psi^{\frac{3p+4}{4-p}}_1=0.
\end{equation}
\end{corollary}
\begin{proof}
The best constant of \eqref{inhomo-embed} is obtained from Theorem
\ref{teo-best}. Then the lower bound \eqref{lower} is derived by a direct
calculation from the proof of Lemma \ref{bozk-ineq} taking into account the
best constant in \eqref{gn1dbest}.
\end{proof}

\section{Proofs of Theorems  \ref{uniform} and \ref{uniform1}}\label{sec4}

As an application of Theorem \ref{teo-best}, we will study the uniform bound
of the solutions to the generalized BO-ZK equation \eqref{bo-zk} stated in
Theorems \ref{uniform} and \ref{uniform1}. We first recall the following
well-posedness result.

\begin{theorem}\label{local}
 Let $s>2$. For any $u_0\in H^s(\rr^2)$, there exists
$T=T(\|u_0\|_{H^s})>0$ and  a unique solution $u\in C([0,T);H^s(\rr^2))$ of
equation \eqref{bo-zk} with $u(0)=u_0$. In addition, $u(t)$ depends
continuously on $u_0$  in the $H^s$-norm. Moreover for all $t\in[0,T)$, we
have $\|u(t)\|_\lt=\|u_0\|_\lt$ and $E(u(t))=E(u_0)$,   where
\begin{equation}\label{third-condition}
E(u)=\frac{1}{2}\int_{\rr^2}\left( u_y^2+ u\h
u_x\right)\,dxdy-\frac{1}{p+2}\int_{\rr^2}u^{p+2}\,dxdy.
\end{equation}
\end{theorem}

Theorem \ref{local} is proved by using the parabolic regularization method
(see \cite{cp} and \cite{aaj}). On the other hand, it was showed in
\cite{pams} that one cannot apply the contraction principle to prove the
local well-posedness of the Cauchy problem associated with \eqref{bo-zk}.
Thus, improvements of Theorem \ref{local} should consider the dispersive
caracter of the equation combined with a compactness-type argument. Note,
however, that Theorems \ref{uniform} and \ref{uniform1}
could be true at any regularity level above the energy space $\ihf$.\\

\noindent\textbf{Proof of Theorem  \ref{uniform}.}\quad Let $u\in
C([0,T);H^s(\rr^2))$ be the solution of \eqref{bo-zk} with the initial data
$u_0\in H^s(\rr^2)$, $s>2$. Then by using the invariants $E$ and
$\|\cdot\|_\lt$, we have
\begin{equation}\label{energy-estimate}\begin{split}
2E(u_0)&=\int_{\rr^2}\left(u_y^2+u\h u_x\right)\,dxdy-\frac{2}{p+2}\int_{\rr^2}u^{p+2}\,dxdy\\
&
\geq\|u\|_\hf^2-\frac{2}{p+2}\|u\|_{L^{p+2}(\rr^2)}^{p+2}\\
&
\geq\|u\|_\hf^2-\frac{2\varrho}{p+2}\|u\|_\lt^{(4-p)/2}\|\dx u\|_\lt^{p}\|u_y\|_\lt^{p/2}\\
& \geq\|u\|_\hf^2-\frac{2\varrho}{p+2}\|u_0\|_\lt^{(4-p)/2}\|u\|_\hf^{3p/2}.
\end{split}\end{equation}
If $0<p<4/3$,  then \eqref{energy-estimate} immediately implies that
$\|u\|_\hf$ (hence $\|u\|_\ihf$) is uniformly bounded for all $t\in [0,T)$.
If $p=4/3$, then we have uniform bound provided that
\begin{equation}\label{12a}
1-\frac{2\vr}{p+2}\|u_0\|_{L^2}^{4/3}>0.
\end{equation}
Using \eqref{besng} we see that \eqref{12a} is equivalent to \eqref{uni134}.
This completes the proof of the theorem. \fim

To prove Theorem \ref{uniform1} we will use the following lemma.

\begin{lemma}\label{begout}
Let $I:=[0,T)\subset\rr$ be a non-degenerated interval. Let $q > 1$, $a>0$,
$b>0$, be real constants. Define  $\vartheta=(bq)^{-1/(q-1)}$ and
$f(r)=a-r+br^q$ for $r\geq0$. Let $G(t)$ be a continuous nonnegative function
on $I$.  If $G(0)<\vartheta$, $a< (1-1/q)\vartheta$ and $f\circ G\geq0$, then
$G(t)<\vartheta$, for any $t\in I$.
\end{lemma}
\begin{proof}
This lemma was essentially established in \cite{begout}. We present here the
minor modifications in the proof. Since $G$ is continuous and
$G(0)<\vartheta$, there exists $0<\varepsilon<T$ such that $G(t)<\vartheta$,
for all $t\in[0,\varepsilon)$. Assume the lemma is false. By the continuity
of $G$ we then deduce the existence of $t^*\in[\varepsilon,T)$ such that
$G(t^*)=\vartheta$. Thus,
\[
f\circ
G(t^*)=f(\vartheta)=a-\vartheta(1-b\vartheta^{q-1})=a-\vartheta\left(1-\frac{1}{q}\right)<0,
\]
which contradicts the fact that $f\circ G\geq0$. The lemma is thus proved.
\end{proof}

\noindent\textbf{Proof of Theorem  \ref{uniform1}.}

In view of \eqref{energy-estimate} and Lemma \ref{begout}, we define
$G(t)=\|u(t)\|_\hf^2$ and $f(r)=a-r+br^q$, where
$$
a=2E(u_0), \quad b=\frac{2\varrho}{p+2}\|u_0\|_\lt^{(4-p)/2}, \quad {\rm
and}\quad q=\frac{3p}{4}.
$$
It follows from Theorem \ref{local} that $G$ is continuous. Moreover, from
\eqref{energy-estimate} we have $f\circ G\geq0$. Thus, the theorem will be
proved if we can show that $G(0)<\vartheta$, $a< (1-1/q)\vartheta$, where
$\vartheta=(bq)^{-1/(q-1)}$.

Now using \eqref{besng} it is not difficult to check that  $G(0)<\vartheta$
is equivalent to \eqref{cond1}. Moreover, using Lemma \ref{lemsapre} we
deduce that
$$
2E(\ff)=\frac{3p-4}{4-p}\|\ff\|^2_\lt.
$$
Hence, $a< (1-1/q)\vartheta$ is equivalent to \eqref{cond2}. Thus, from Lemma
\ref{begout} we have $G(t)<\vartheta$, which in turn is equivalent to
\eqref{cond3}.

Hence, it is deduced from $\|u(t)\|_\lt=\|u_0\|_\lt$, for all $t\in[0,T)$,
that $u(t)$ is uniformly bounded in $\ihf$ for all $[0,T)$. \fim

\begin{remark}
Note that in the limiting case $p=4/3$, conditions \eqref{cond1} and
\eqref{cond2} in Theorem \ref{uniform1} reduce to the same one, which is
exactly condition \eqref{uni134} in Theorem \ref{uniform}.
\end{remark}

\section*{Acknowledgment}
The first author is partially supported by   a grant from IPM (No. 92470042).
The second author is partially supported by CNPq-Brazil and FAPESP-Brazil.

\end{document}